\documentclass[
    11pt,a4paper,twoside,   
    reqno,                  
    hidelinks,              
    dvipsnames,             
    ]{amsart}

\usepackage[margin=0.9in]{geometry} 
\usepackage[indent]{parskip}        

\usepackage{amsmath} 
\usepackage{amsfonts}
\usepackage{amssymb}    
\usepackage{amsthm}     
\usepackage{mathtools}  
\usepackage{dsfont}     

\newcommand{\N}{\mathbb N}
\newcommand{\Z}{\mathbb Z}
\newcommand{\R}{\mathbb R}

\renewcommand{\leq}{\leqslant}
\renewcommand{\geq}{\geqslant}

\renewcommand{\epsilon}{\varepsilon}
\renewcommand{\phi}{\varphi}

\newcommand{\defeq}{\coloneqq} 
\newcommand{\st}{\ s.t.\ } 
\newcommand{\divides}{\vert} 
\newcommand{\notdivides}{\nmid}
\newcommand{\norm}[1]{\lVert #1 \rVert} 
\newcommand{\abs}[1]{\lvert #1 \rvert} 

\newcommand{\indicator}[1]{\mathds 1_{#1}} 

\DeclareMathOperator{\supp}{supp} 
\DeclareMathOperator{\Log}{Log} 

\theoremstyle{plain}
\newtheorem{theorem}{Theorem}[section]
\newtheorem{lemma}[theorem]{Lemma}
\newtheorem{proposition}[theorem]{Proposition}
\newtheorem{corollary}[theorem]{Corollary}

\theoremstyle{definition}
\newtheorem{definition}[theorem]{Definition}

\numberwithin{equation}{section} 
\mathtoolsset{showonlyrefs=true} 

\usepackage{xcolor} 

\usepackage{hyperref} 
\hypersetup{
    colorlinks=true,        
    linkcolor=teal,         
    citecolor=ForestGreen,  
}

\begin{document}

\title{Almost-sharp quantitative Duffin--Schaeffer without GCD graphs}
\author{Santiago Vazquez}
\address{King's College London, Department of Mathematics, WC2R 2LS London, United Kingdom}
\email{santiago.vazquez\_saez@kcl.ac.uk}

\begin{abstract}
    In recent work, Koukoulopoulos, Maynard and Yang proved an almost sharp quantitative bound for the Duffin--Schaeffer conjecture, using the Koukoulopoulos--Maynard technique of GCD graphs. This coincided with a simplification of the previous best known argument by Hauke, Vazquez and Walker, which avoided the use of the GCD graph machinery. In the present paper, we extend this argument to the new elements of the proof of Koukoulopoulos--Maynard--Yang. Combined with the work of Hauke--Vazquez--Walker, this provides a new proof of the almost sharp bound for the Duffin--Schaeffer conjecture, which avoids the use of GCD graphs entirely.
\end{abstract}
\maketitle


\section{Introduction}
The Duffin--Schaeffer conjecture \cite{DS41} is concerned with the proportion of real numbers $\alpha$ for which there are infinitely many reduced fractions $\frac{a}{q}$, satisfying
\begin{align}
    \label{eq:rational_approx}
    \Big\lvert \alpha -\frac{a}{q} \Big\rvert \leq \frac{\psi(q)}{q}.
\end{align}
Here, $\psi:\N\to\R_{\geq 0}$ is a fixed function, and the conjecture states that this proportion is 1 (in the sense of the Lebesgue measure) if 
\begin{align}
    \label{eq:divergence_condition}
    \sum_{q\geq 1} \frac{\psi(q)\phi(q)}{q} =\infty,
\end{align}
and it is 0 otherwise.

This question was settled by Koukoulopoulos and Maynard in \cite{KM20}, where they proved
\begin{theorem}
    [Theorem 1, \cite{KM20}]
    \label{thm:km}
    Let $\psi:\N\to\R_{\geq 0}$ be a function satisfying \eqref{eq:divergence_condition}. Then for almost all $\alpha \in [0,1]$ there are infinitely many coprime solutions $(a,q)$ to \eqref{eq:rational_approx}.
\end{theorem}
Later work of Aistleitner, Borda and Hauke established a quantitative version of this result.
\begin{theorem}
    [Theorem 1, \cite{ABH23}]
    \label{thm:abh}
    Let $\psi:\N\to[0,1/2]$ be a function satisfying \eqref{eq:divergence_condition} and define 
    \begin{align}
        \Psi(Q) \defeq \sum_{q\leq Q} \frac{2\psi(q)\phi(q)}{q}.
    \end{align}
    Write $S(\alpha, Q)$ for the number of coprime solutions $(a,q)$ to \eqref{eq:rational_approx} with $q\leq Q$.
    Then for any $A>0$ and almost all $\alpha$,
    \begin{align}
        S(\alpha, Q) = \Psi(Q)\Big( 1 + O_{\alpha, A} ((\log \Psi (Q))^{-A}) \Big).
    \end{align}
\end{theorem}

Both of their arguments rely heavily on the technical formalism of GCD graphs. However, in recent work of Hauke, Walker and the author \cite{HVW24}, a shorter proof of Theorem \ref{thm:km} was given which avoided the use of GCD graphs entirely. Instead, it follows a strategy based on the work of Green and Walker \cite{GW21}, introducing some conceptual simplifications,
to prove a generalization of the main technical step \cite[Proposition 5.4]{KM20} and \cite[Proposition 7]{ABH23}.
This method also improves the error term of Theorem \ref{thm:abh}, 
\begin{align}
    S(\alpha, Q) = \Psi(Q)\Big( 1+O_{\alpha, \epsilon}(\exp(-(\log \Psi(Q))^{\frac{1}{2} -\epsilon})) \Big).
\end{align}

Independently of the work in \cite{HVW24}, but simultaneously, Koukoulopoulos, Maynard and Yang \cite{KMY24} improved Theorem \ref{thm:abh}, obtaining an almost sharp bound.
\begin{theorem}
    [Theorem 1, \cite{KMY24}]
    \label{thm:kmy}
    Let $\psi, \Psi$ and $S$ be as in Theorem \ref{thm:abh}, then for any $\epsilon > 0$ and almost all $\alpha$, 
    \begin{align}
        S(\alpha, Q) = \Psi(Q) + O_ \epsilon(\Psi(Q)^{\frac{1}{2} + \epsilon})
    \end{align}
    for every $Q$ large enough in terms of $\alpha$ and $\psi$.
\end{theorem}

Their work again utilizes GCD graphs, together with an improved overlap estimate \cite[Proposition 5.1]{KMY24}.
They reduce the proof of Theorem \ref{thm:kmy} to three propositions: Propositions 7.1, 7.2 and 7.3 in \cite{KMY24}, which they prove using GCD graphs.
The first two correspond with the main technical step mentioned above, and are proved in \cite{HVW24} without GCD graphs (see the remark below for a small technicality). Proposition 7.3, on the other hand, concerns a different anatomic condition over the small prime factors, and is not considered in \cite{HVW24}.

The goal of this paper is to show that the argument in \cite{HVW24} can nonetheless be adapted to prove a version of \cite[Proposition 7.3]{KMY24}. 
In fact, we generalise this bound to a wide class of similar inequalities.
This, in combination with \cite{HVW24}, gives a shorter proof of Theorem \ref{thm:kmy} which avoids the complexity of GCD graphs entirely.

Throughout the rest of the paper, we write 
\begin{align}
    \omega_t(v,w) \defeq \big\lvert \big\{ p\leq t : p\divides \frac{vw}{\gcd(v,w)} \big\}\big\rvert
    \qquad \text{and} \qquad
    D_{\psi, \theta}(v,w) \defeq \frac{\operatorname{max}(w\psi(v),v\theta(w))}{\gcd(v,w)},
\end{align}
for any pair $v,w\in\N$, real $t\geq 1$ and any functions $\psi,\theta :\N\to\R_{\geq 0}$.
Let us now state the relevant proposition from \cite{KMY24}.

\begin{proposition}
    [Proposition 7.3 in Koukoulopoulos--Maynard--Yang
    \cite{KMY24}]
    \label{prop:kmy}
    Fix $\epsilon, \kappa > 0$ and $C\geq 1$. Let $\psi:\N\to\R_{\geq 0}$ and $y,t\geq 1$. We then have the uniform estimate 
    \begin{align}
        \mathop{\sum\sum}_{\substack{(q,r)\in [1,Q]^{2}\\ D_{\psi, \psi}(q,r)\leq y,\, \omega_t(q,r)\geq\kappa\log t}}\frac{\psi(q)\phi(q)}{q}\cdot \frac{\psi(r)\phi(r)}{r} \ll_{\epsilon, \kappa, C}t ^{-C}y^{1-\epsilon}\Psi(Q)^{1+\epsilon}.
    \end{align}
\end{proposition}

\textbf{Remark.}
There is a slight issue with the current statement of \cite[Proposition 7.3]{KMY24}, as it may fail when $\Psi(Q)<1$ and $\epsilon$ tends to infinity. This can be fixed by requiring $\epsilon$ to be small, as both the proof and the applications work when $\epsilon$ is small enough.

Before we state our main result, we define the notation introduced in \cite{HVW24} for the weight associated with certain sets of integers and pairs of integers.
\begin{definition}
    [Measures]
    \label{def:measures}
    Let $\psi, \theta: \N \to \R_{\geq 0}$ be finitely supported, and let $f, g: \N \to \R_{\geq 0}$ be multiplicative functions.
    For $v \in \N$ we define 
    \begin{align}
        \mu_\psi^{f}(v) \defeq \frac{f(v)\psi(v)}{v}.
    \end{align}
    When $V\subset \N$ we define 
    \begin{align}
        \mu_{\psi}^{f}(V) \defeq \sum_{v\in V}\mu_\psi^{f}(v)
        = \sum_{v\in V} \frac{f(v)\psi(v)}{v}.
    \end{align}
    If $\mathcal E \subset \N\times\N$, we define 
    \begin{align}
        \mu_{\psi, \theta}^{f,g}(\mathcal E) \defeq \sum_{(v,w) \in \mathcal E}\mu_\psi^{f}(v)\mu_\theta^{g}(w).
    \end{align}
\end{definition}
We also introduce a special set of pairs $(v,w)$ of integers with small $D_{\psi,\theta}(v,w)$ and large $\omega_t(v,w)$.
\begin{definition}
    Let $\psi, \theta : \N \to \R_{\geq 0}$ be finitely supported, with $V_{\psi} = \supp\psi$ and $W_{\theta} = \supp\theta$, and let $t\geq 1$ and $K \in \R$. Then we define 
    \begin{align}
        \mathcal E_{\psi, \theta}^{t, K} \defeq \{ (v,w) \in V_\psi\times W_\theta: D_{\psi,\theta}(v,w) \leq 1, \omega_t(v,w) \geq K \}.
    \end{align}
\end{definition}

Finally, throughout this paper, we use the notation
\begin{align}
    \Log t \defeq \max\{ 1, \log t \}
\end{align}
for $t\geq 1$.

We are now ready to state our main result.

\begin{theorem}
    \label{thm:main}
    Let $\epsilon \in (0, 2/5]$ and $C > 0$. Then there exists $p_0(\epsilon, C)>0$ such that the following holds. Let $\psi, \theta: \N\to\R_{\geq0}$ be finitely supported, $V_{\psi} = \supp\psi$, $W_\theta = \supp\theta$ and 
    \begin{align}
        \mathcal P_{\psi,\theta} \defeq \{ p: \exists(v,w)\in V_{\psi}\times W_{\theta} \st p\divides vw \}.
    \end{align}
    Let $P_{\psi, \theta} \defeq p_0(\epsilon, C) + \lvert \mathcal P_{\psi, \theta}\cap [1, p_0(\epsilon, C)]\rvert$.
    Let $f, g:\N\to\R_{\geq 0}$ be multiplicative functions for which 
    \begin{align}
        \label{eq:f-condition}
        (1\star f)(n)\leq n, \qquad \text{and} \qquad (1\star g)(n)\leq n \qquad (\text{for all $n\geq 1$}).
    \end{align}
    Suppose that $\mathcal E \subset \mathcal E_{\psi,\theta}^{t, K}\cap(V\times W)$. Then for all $t\geq 1$ and any $K\in \R$ we have
    \begin{align}\label{eq:main-thm-bound}
        \mu_{\psi,\theta}^{f,g}(\mathcal E) \leq (100e^{C}) ^{P_{\psi,\theta}(\epsilon, C)}(\Log t)^{\frac{1}{2}(e^{40C}-1)}(\mu_{\psi}^{f}(V)\mu_{\theta}^{g}(W)e^{-CK})^{\frac{1}{2}+\epsilon}.
    \end{align}
\end{theorem}
The above theorem relates back to the Duffin--Schaeffer situation when $f = g = \phi$.

\begin{corollary}
    \label{cor:main_application}
    Let $\epsilon \in (0,4/5]$, $C >0$, $t\geq 1$ and $K\in\R$. Let also $\psi:\N\to\R_{\geq 0}$ be finitely supported with $V =\supp \psi$. Then we have 
    \begin{align}
        \mu_{\psi,\psi}^{\phi,\phi}(\mathcal E_{\psi,\psi}^{t, K}) \ll_{\epsilon, C} (\Log t)^{\frac{1}{2}(e^{80C}-1)}e^{-CK(1+\epsilon)}\mu_{\psi}^{\phi}(V)^{1+\epsilon}.
    \end{align}
\end{corollary}
\begin{proof}
    Observe that $\phi$ satisfies \eqref{eq:f-condition}. Also, $P_{\psi,\theta}(\epsilon, C)$ can be bounded above in terms of $p_0(\epsilon, C)$, which is independent of $\psi$ and $\theta$. The result now follows immediately from Theorem \ref{thm:main}, where we use $\epsilon/2$ in place of $\epsilon$.
\end{proof}

\textbf{Remarks.}
\begin{enumerate}
    \item In the statement of Theorem \ref{thm:main}, the upper bound $2/5$ in the range of $\epsilon$ can be replaced by any constant below $1/2$ at the cost of substituting the 100 by a larger constant. Moreover, it is possible to get the 40 in the exponent of $\Log t$ arbitrarily close to $4$, provided that $\epsilon$ is small enough.
    \item The quantity $P_{\psi,\theta}(\epsilon, C)$ is introduced to make the later inductive argument work. In practice, one may substitute the leading factor by an implicit constant depending on $\epsilon$ and $C$, as in Corollary \ref{cor:main_application}.
    \item Theorem \ref{thm:main} introduces an arbitrary constant $C$, which was not present in the previous work \cite[Theorem 1.7]{HVW24}. However, Propositions 7.1, 7.2 and 7.3 from \cite{KMY24} do include the same arbitrary constant. The presence of this constant in our main theorem introduces two minor technical changes with respect to the argument in \cite{HVW24}. The first is a slightly more complicated leading factor that accounts for the $C$ dependency. The second is the introduction of a similar arbitrary constant in the anatomy lemmas from Section \ref{sec:anatomy_lemmas}. This is achieved by not making any concrete choices of constants, such as the 100 in Lemmas 4.1 and 4.2 from \cite{HVW24}. The treatment of both of these changes in the present paper transfers immediately to \cite{HVW24}, thus proving a version of the main technical theorem where the same arbitrary constant is present.
\end{enumerate}

\begin{proof}
    [Proof of Proposition \ref{prop:kmy} from Theorem \ref{thm:main} for small $\epsilon$]

    Let $\epsilon, \kappa, C, \psi, y$ and $t$ be as in the statement of Proposition \ref{prop:kmy}.
    Observe that Theorem \ref{thm:main} only applies to the sum of Proposition \ref{prop:kmy} when $y = 1$. We can get around this issue by rescaling the function $\psi$. We define 
    \begin{align}
        \widetilde\psi(n) \defeq y^{-1}\psi(n)\cdot \indicator{n\leq Q}(n).
    \end{align}
    Proposition \ref{prop:kmy} is then equivalent to 
    \begin{align}
        \label{eq:kmy-equivalent}
        \mu_{\widetilde\psi,\widetilde\psi}^{\phi, \phi}( \mathcal E_{\widetilde\psi, \widetilde\psi}^{t,\kappa\log t} )\ll_{\epsilon, \kappa, C} t ^{-C}\mu_{\widetilde\psi}^{\phi}([Q])^{1+\epsilon},
    \end{align}
    where $[Q]$ denotes the set of positive integers $n$ with $n\leq Q$.

    We assume that $\epsilon\leq 4/5$ (see remark above).
    Let $\widetilde C >0$ be a positive real constant to be fixed later.
    Then by Corollary \ref{cor:main_application} with $\widetilde C$ in place of $C$ and $\widetilde \psi$ in place of $\psi$,
    \begin{align}
        \mu_{\widetilde\psi,\widetilde\psi}^{\phi, \phi}(\mathcal E_{\widetilde\psi,\widetilde\psi}^{t, \kappa\log t})
        &\ll_{\epsilon, \widetilde C} t ^{-\widetilde C\kappa(1+\epsilon)}(\Log t)^{\frac{1}{2}(e^{80\widetilde C}-1)}
        \mu_{\widetilde\psi}^{\phi}([Q])^{1+\epsilon}\\
        &\ll_{\epsilon,C, \widetilde C} t^{-\widetilde C\kappa +C
        }
        \mu_{\widetilde\psi}^{\phi}([Q])^{1+\epsilon}.
    \end{align}
    The bound \eqref{eq:kmy-equivalent} now follows by setting $\widetilde C = \frac{2C}{\kappa}$.
\end{proof} 

The remainder of this paper is concerned with the proof of Theorem \ref{thm:main}.
In Section \ref{sec:proof_strategy}, we lay down the strategy of the proof, reducing it to two propositions (Proposition \ref{prop:structure-of-counterexample} and Proposition \ref{prop:resolution-of-counterexample}).
Section \ref{sec:proof_of_structure} is devoted to the proof of Proposition \ref{prop:structure-of-counterexample}.
In Section \ref{sec:anatomy_lemmas}, we prove two auxiliary lemmas concerned with the anatomic condition of $\omega_t(v,w)$ being large. In Section \ref{sec:proof_of_resolution}, we deduce Proposition \ref{prop:resolution-of-counterexample} from the anatomy lemmas of the previous section, thus completing the proof of Theorem \ref{thm:main}.

\textbf{Acknowledgements.}
The author thanks Aled Walker for invaluable guidance and comments on an earlier version of the paper.
This work was supported by the Engineering and Physical Sciences Research Council [EP/S021590/1]. 
The EPSRC Centre for Doctoral Training in Geometry and Number Theory (The London School of Geometry and Number Theory), University College London.

\textbf{Notation.} We use the standard $O$-notation as well as the Vinogradov notations $\ll, \gg$.
For a set $\mathcal E \subset V\times W$ of pairs and an element $v\in V$, we write $\Gamma_{\mathcal E}(v)\defeq \{ w\in W: (v,w)\in \mathcal E \}$ and $\mathcal E\vert_{V} \defeq \{ v\in V: \Gamma_{\mathcal E}(v)\neq \emptyset \}$.
For $w\in W$, $\Gamma_{\mathcal E}(w)$ and $\mathcal E\vert_{W}$ are defined analogously.
The set of positive integers not exceeding $N$ is denoted by $[N]$.
For a prime $p$, we denote by $\nu_p$ the usual $p$-adic valuation over the rationals. Given multiplicative functions $f,g$ we write $(f\star g)(n)\defeq \sum_{ab = n}f(a)g(b)$ for their Dirichlet convolution.
When $\epsilon \in (0,1/2)$ is given, we will write $q\defeq \frac{2}{1-2\epsilon}$ and $q'\defeq \frac{2}{1+2\epsilon}$. We write $\Log t \defeq \max\{ 1, \log t \}$ for $t\geq 1$.

\section{Proof strategy}
\label{sec:proof_strategy}
The proof of Theorem \ref{thm:main} will closely follow the strategy laid out in \cite{HVW24}, which naturally splits into the two propositions below. The first part of the argument considers a potential counterexample, subject to some minimality condition, and finds that it must contain a large and highly structured subset.
For the second part, we prove that such a structured set cannot be of the size obtained in the first part.

\begin{proposition}
    [Structure of minimal potential counterexample]
    \label{prop:structure-of-counterexample}
    Suppose Theorem \ref{thm:main} were false. Fix functions $\psi, \theta: \N\to\R_{\geq 0}$ with finite support such that $\lvert\mathcal P_{\psi,\theta}\rvert$ is minimal over all such pairs of functions for which there exist instances of $\epsilon, C, t, K, f, g$, and $\mathcal E\subset \mathcal E_{\psi,\theta}^{t, K}$ satisfying the hypotheses of Theorem \ref{thm:main} for which \eqref{eq:main-thm-bound} fails. Fix such instances of $\epsilon, C, t, K, f, g,$ and $\mathcal E$. Write $q' \defeq \frac{2}{1+2 \epsilon}$. Then for $p_0(\epsilon, C)$ large enough in terms only of $\epsilon$ and $C$, there exists $\mathcal E'\subset \mathcal E$ for which

    \begin{enumerate}
        \item\label{itm:is-near-counterexample}
            \emph{($\mathcal E'$ is a near counterexample)}: 
            \begin{align}
                \mu_{\psi,\theta}^{f,g}(\mathcal E') > \frac{1}{2} (100e^{C}) ^{P_{\psi,\theta}(\epsilon, C)}(\Log t)^{\frac{1}{2}(e^{40C}-1)}(\mu_{\psi}^{f}(V')\mu_{\theta}^{g}(W')e^{-CK})^{\frac{1}{q'}}
            \end{align}
            where $V'\defeq \mathcal E'\rvert_{V}, W'\defeq \mathcal E'\rvert_{W}$.
        \item\label{itm:combinatorially-structured}
            \emph{($\mathcal E'$ is combinatorially structured)}: for all $v\in V'$ and $w\in W'$, 
            \begin{align}
                \mu_{\theta}^{g}(\Gamma_{\mathcal E'}(v))\geq \frac{1}{q'}\frac{\mu_{\psi,\theta}^{f,g}(\mathcal E')}{\mu_{\psi}^{f}(V')}
                \qquad\text{and}\qquad
                \mu_{\psi}^{f}(\Gamma_{\mathcal E'}(w))\geq \frac{1}{q'}\frac{\mu_{\psi,\theta}^{f,g}(\mathcal E')}{\mu_{\theta}^{g}(W')}.
            \end{align}
        \item\label{itm:arithmetically-structured}
\emph{($\mathcal E'$ is arithmetically structured)}: there exists $N\in\N$ such that for all primes $p$ and for all $(v,w)\in\mathcal E'$, $\lvert\nu_p(v/N)\rvert+\lvert\nu_p(w/N)\rvert \leq 1$.
    \end{enumerate}
\end{proposition}

\begin{proposition}
    [Resolution of minimal potential counterexample]
    \label{prop:resolution-of-counterexample}
    Fix $\psi, \theta, \epsilon, C, t, K, f, g$ satisfying the hypotheses of Theorem \ref{thm:main}.
    Suppose that $\mathcal E'\subset\mathcal E_{\psi,\theta}^{t, K}$ satisfies properties \ref{itm:combinatorially-structured} and \ref{itm:arithmetically-structured} from Proposition \ref{prop:structure-of-counterexample} (with these parameters).
    Then $\mathcal E'$ cannot satisfy property \ref{itm:is-near-counterexample} of Proposition \ref{prop:structure-of-counterexample}.
\end{proposition}
The combination of Propositions \ref{prop:structure-of-counterexample} and \ref{prop:resolution-of-counterexample} immediately implies Theorem \ref{thm:main}.

It is remarkable how well the method from \cite{HVW24} adapts to this new situation. And indeed, aside from the difference in the bounds coming from the anatomy lemmas, Lemmas \ref{lm:unweighted_anatomy} and \ref{lm:divisor_anatomy}, no new insight is needed for the proof. We refer the reader to the original paper \cite{HVW24} for more in-depth explanations and remarks.

\section{Proof of Proposition \ref{prop:structure-of-counterexample}}
\label{sec:proof_of_structure}
Let $\psi, \theta, \epsilon, C, t, K, f, g$ and $\mathcal E$ be the potential counterexample to Theorem \ref{thm:main} defined in the statement of Proposition \ref{prop:structure-of-counterexample}, and let $p\in\mathcal P_{\psi,\theta}$. We will show that the minimality condition of the counterexample implies that the $p$-adic valuations $(\nu_p(v), \nu_p(w))$ for $(v,w)\in\mathcal E$ concentrate near a diagonal point. For integers $i,j\geq 0$, define $V_i\defeq \{ v\in V:\nu_p(v)=i \}$ and $W_j\defeq\{w\in W:\nu_p(w)=j\}$. We also write 
\begin{align}
    m(i,j)\defeq \frac{\mu_{\psi,\theta}^{f,g}( \mathcal E\cap(V_i\times W_j) )}{\mu_{\psi,\theta}^{f,g}( \mathcal E )}.
\end{align}
This is well defined as $\mu_{\psi,\theta}^{f,g}( \mathcal E )>0$ since $\mathcal E$ is a counterexample to Theorem \ref{thm:main}.

The first step is to use the minimality assumption to prove a bilinear upper bound for $m(i,j)$.
\begin{lemma}
    [Bilinear upper bound]
    \label{lm:bilinear-upper-bound}
    Let $\psi, \theta, \epsilon, C, t, K, f, g, \mathcal E, m(i,j)$ be as above. Writing for $i,j\geq 0$, $\alpha_i \defeq\mu_\psi^{f}(V_{i})/\mu_\psi^{f}(V)$ and $\beta_j \defeq \mu_\theta^{g}(W_j)/\mu_\theta^{g}(W)$, we have 
    \begin{align}
        m(i,j)\leq
        (100e^{C})^{-\indicator{p\leq p_0(\epsilon, C)}}p^{\frac{-\lvert i-j\rvert}{q}}(\alpha_i \beta_j e^{\indicator{i\neq j}\cdot C})^{\frac{1}{q'}}.
    \end{align}
\end{lemma}
\begin{proof}
    We will bound $m(i, j)$ by bounding $\mu _{\psi, \theta}^{f, g}(\mathcal E)$ from below and $\mu_{\psi,\theta}^{f, g}(\mathcal E\cap(V_i\times W_j))$ from above.
    For the lower bound we simply use the assumption that $\mathcal E$ is a counterexample to Theorem \ref{thm:main}, as \eqref{eq:main-thm-bound} implies
    \begin{align}
        \label{eq:lower-bound-bub-lemma}
        \mu_{\psi,\theta}^{f,g}(\mathcal E) \geq (100e^{C}) ^{P_{\psi,\theta}(\epsilon, C)}(\Log t)^{\frac{1}{2}(e^{40C}-1)}(\mu_{\psi}^{f}(V)\mu_{\theta}^{g}(W)e^{-CK})^{\frac{1}{2}+\epsilon}.
    \end{align}

    For the upper bound, we remove the contribution of the prime $p$. By the minimality condition of $\mathcal E$, the result will satisfy Theorem \ref{thm:main}, which gives the upper bound.
    We define
    \begin{align}
        \widetilde{\mathcal E_{i, j}} \defeq \{ (v, w):(p^{i}v, p^{j}w)\in\mathcal E\cap(V_i\times W_j) \},\quad
        \widetilde{V_i} \defeq \{ v: p^{i}v\in V_i \}
        \quad\text{and}\quad
        \widetilde{W_j} \defeq \{ q: p^{j}w\in W_j \}.
    \end{align}
    We also introduce scaled versions of $\psi$ and $\theta$ as follows, 
    \begin{align}
        \widetilde{\psi_{i, j}}(v) \defeq
        \begin{cases}
            p^{j -\min (i,j)}\psi(p^{i}v)   &\text{if } p\notdivides v,\\
            0                               &\text{if } p\divides v,
        \end{cases}
    \end{align}
    and
    \begin{align}
        \widetilde{\theta_{i, j}}(w) \defeq
        \begin{cases}
            p^{i -\min (i,j)}\theta(p^{j}w) &\text{if } p\notdivides w,\\
            0                               &\text{if } p\divides w.
        \end{cases}
    \end{align}
    The rescaling guarantees that for $v\in \widetilde{V_i}$ and $w\in\widetilde{W_j}$, we have $D_{\widetilde{\psi_{i, j}}, \widetilde{\theta_{i, j}}}(v, w) = D_{\psi, \theta}(p^{i}v, p^{j}w) \leq 1$. 
    Moreover, it is clear that $\omega_{t}(v, w) = \omega_{t}(p^{i}v, p^{j}w)-\indicator{\substack{i\neq j\\ p\leq t}} \geq K-\indicator{i\neq j}$ whenever $v\in \widetilde{V_i}$ and $w\in\widetilde{W_j}$.
    Therefore, $\widetilde{\mathcal E_{i,j}} \subset \mathcal E_{\widetilde{\psi_{i, j}}, \widetilde{\theta_{i, j}}}^{t, K-\indicator{i\neq j}}$ and so $\widetilde{\mathcal E_{i,j}}$ is of the form considered in Theorem \ref{thm:main}.
    Finally, 
    \begin{align}
        \mathcal P_{\widetilde{\psi_{i, j}}, \widetilde{\theta_{i, k}}} \subseteq \mathcal P_{\psi,\theta}\setminus\{p\},
    \end{align} 
    and by the minimality assumption of the counterexample $\mathcal E$, the bound \eqref{eq:main-thm-bound} holds for $\mu_{\widetilde{\psi_{i, j}},\widetilde{\theta_{i, j}}}^{f, g}(\widetilde{\mathcal E_{i, j}})$. Hence 
    \begin{align}
        \label{eq:inductive-bound-no-p}
        \mu_{\widetilde{\psi_{i, j}},\widetilde{\theta_{i, j}}}^{f, g}(\widetilde{\mathcal E_{i, j}})
        \leq (100e^{C})^{P_{\psi, \theta}(\epsilon, C)-\indicator{p\leq p_0}}(\Log t)^{\frac{1}{2}(e^{40C}-1)}
        (
            \mu_{\widetilde{\psi_{i, j}}}^{f}(\widetilde{V_{i}})
            \mu_{\widetilde{\theta_{i, j}}}^{g}(\widetilde{W_{j}})
            e^{-C(K-\indicator{i\neq j})}
        )^{\frac{1}{2}+\epsilon}.
    \end{align}
    It is now a matter of unpacking the definitions. Observe that for $v\in \widetilde{V_{i}}$,
    \begin{align}
        \frac{f(v)\widetilde{\psi_{i,j}}(v)}{v}
        = p^{j-\min(i, j)}\Big( \frac{p^{i}}{f(p^{i})} \Big)\frac{f(p^{i}v)\psi(p^{i}v)}{p^{i}v}
    \end{align}
    and therefore
    \begin{align}
        \mu_{\widetilde{\psi_{i,j}}}^{f}(\widetilde{V_{i}})
        = p^{j-\min(i, j)}\Big( \frac{p^{i}}{f(p^{i})} \Big) \mu_{\psi}^{f}(V_i).
    \end{align}
    Similarly 
    \begin{align}
        \mu_{\widetilde{\theta_{i,j}}}^{g}(\widetilde{W_{j}})
        = p^{i-\min(i, j)}\Big( \frac{p^{j}}{g(p^{j})} \Big) \mu_{\theta}^{g}(W_j),
    \end{align}
    and also 
    \begin{align}
        \mu_{\widetilde{\theta_{i,j}}, \widetilde{\psi_{i,j}}}^{f,g}(\widetilde{\mathcal E_{i, j}})
        = \Big( \frac{p^{i+j}}{f(p^{i})g(p^{j})} \Big) p^{\lvert i-j\rvert} \mu_{\psi, \theta}^{f,g}(\mathcal E\cap(V_i\times W_j)).
    \end{align}
    Combining this with \eqref{eq:inductive-bound-no-p}, 
    \begin{align}
        &\mu_{\psi, \theta}^{f,g}(\mathcal E\cap(V_i\times W_j))\\
        &\leq \Big( \frac{f(p^{i})g(p^{j})}{p^{i+j}} \Big) ^{\frac{1}{2} - \epsilon}p^{-\lvert i-j \rvert(\frac{1}{2}-\epsilon)}
        (100e^{C})^{P_{\psi,\theta}(\epsilon, C)-\indicator{p\leq p_0(\epsilon, C)}}\\
        &\omit$ \hfill \cdot(\Log t)^{\frac{1}{2}(e^{40C}-1)}(\mu_{\psi}^{f}(V_i)\mu_{\theta}^{g}(W_j)e^{-C(K-\indicator{i\neq j})})^{\frac{1}{2}+\epsilon}$\\
        &\leq p^{-\lvert i-j \rvert(\frac{1}{2}-\epsilon)}
        (100e^{C})^{P_{\psi,\theta}(\epsilon, C)-\indicator{p\leq p_0(\epsilon, C)}}(\Log t)^{\frac{1}{2}(e^{40C}-1)}(\mu_{\psi}^{f}(V_i)\mu_{\theta}^{g}(W_j)e^{-C(K-\indicator{i\neq j})})^{\frac{1}{2}+\epsilon}
    \end{align}
    since $f(p^{i})\leq p^{i}$ and $g(p^{j})\leq p^{j}$.
    Combining the above with the lower bound for $\mu_{\psi, \theta}^{f, g}(\mathcal E)$ from \eqref{eq:lower-bound-bub-lemma} we get 
    \begin{align}
        m(i, j)
        \leq (100e^{C})^{-\indicator{p\leq p_0(\epsilon, C)}}p^{-\lvert i-j\rvert(\frac{1}{2}-\epsilon)}(\alpha_i\beta_j e^{C\cdot \indicator{i\neq j}})^{\frac{1}{2}+\epsilon},
    \end{align}
    as we wanted.
\end{proof}

The concentration near a diagonal point mentioned above follows from the bilinear upper bound by means of the following lemma.
We refer the reader to \cite[Section 3]{HVW24} for a proof and further discussion of this result.
\begin{lemma}
    [Decay away from the diagonal, {\cite[Lemma 3.2]{HVW24}}] 
    \label{lm:decay-away-diagonal}
    Let $q > 2$, and write $q'$ for the conjugate index of $q$ (i.e. $\frac{1}{q}+\frac{1}{q'} = 1$). Let $c_1\leq 1, 0<c_2<1, \lambda \in (0, 1-c_2]$, and $C_3 > 0$. Suppose that $m$ is a finitely-supported probability measure on $\Z^{2}$. Suppose that there are sequences $x = (x_i)_{i\in\Z}, y = (y_j)_{j\in\Z}$ of non-negative reals such that $\lVert x\rVert_{\ell^{q'}(\Z)}=\lVert y\rVert_{\ell^{q'}(\Z)} =1$, and such that for all $(i,j)\in\Z^{2}$ we have 
    \begin{align}
        m(i,j)\leq
        \begin{cases}
            c_1x_iy_j                                   & \text{if }i=j,\\
            c_1C_3\lambda^{\lvert i-j\rvert}x_iy_j      & \text{if }i\neq j.
        \end{cases}
    \end{align}
    Then $c_1 \geq \frac{c_2}{1+(2C_3-1)\lambda}$, and there exists $k\in \Z$ such that 
    \begin{align}
        \sum_{\lvert i-k\rvert + \lvert j-k\rvert \geq 2} m(i,j) \ll_{q,c_2,C_3}(\lambda^{q})^{\frac{2}{q'}}+(\lambda^{q})^{1+\frac{1}{q}}.
    \end{align}
\end{lemma}

We are now ready to complete the proof of Proposition \ref{prop:structure-of-counterexample}.

\begin{proof}
    [Proof of Proposition \ref{prop:structure-of-counterexample}]
    We now apply Lemma \ref{lm:decay-away-diagonal} to the bound obtained in Lemma \ref{lm:bilinear-upper-bound}.
    As usual $q = \frac{2}{1-2 \epsilon}$ and $q'=\frac{2}{1+2 \epsilon}$, and in this case $c_1 = (100e^{C})^{-\indicator{p\leq p_0(\epsilon, C)}}$, $\lambda = p^{-\frac{1}{2}+\epsilon}$ and $C_3 = e^{C}$. Since $\epsilon\leq 2/5$ we may take $c_2 = 1-2^{-\frac{1}{2} + \frac{2}{5}}$. We also take $x_i = \alpha_i^{\frac{1}{2}+\epsilon}$ and $y_j = \beta_j^{\frac{1}{2}+\epsilon} $.
    Extending $x_i, y_j$ and $m(i,j)$ to negative integers by extending by 0, and since $\sum_i \alpha_i = \sum_j \beta_j = 1$, these sequences satisfy the hypothesis $\norm x_{\ell^{q'}(\Z)} = \norm y_{\ell^{q'}(\Z)} = 1$.

    The first conclusion of Lemma \ref{lm:decay-away-diagonal} then implies 
    \begin{align}
        (100e^{C})^{-\indicator{p\leq p_0(\epsilon, C)}} = c_1 \geq \frac{c_2}{1+(2C_3-1)\lambda}\geq \frac{c_2}{2C_3}\geq \frac{1-2^{-1/10}}{2e^{C}} > (100e^{C})^{-1}.
    \end{align}
    Therefore $p > p_0(\epsilon ,C)$. Now by the second conclusion there exists some $k_p\in \Z_{\geq 0}$ such that 
    \begin{align}
        \label{eq:concentration_kp}
        \sum_{\abs{i-k_p}+\abs{j-k_p}\geq 2} m(i,j)\ll_{\epsilon, C}\lambda^{\frac{2q}{q'}}+\lambda^{2+\frac{q}{q'}} = p^{-1-2 \epsilon}+ p^{-\frac{3}{2}+ \epsilon} \leq p^{-1-2 \epsilon} + p^{-\frac{11}{10}} 
    \end{align}
    as $\epsilon \leq 2/5$.

    The final manipulations are now very similar to those in the proof of Proposition 2.1 in \cite[Section 3]{HVW24}.
    Consider all primes $p\in \mathcal P_{\psi, \theta}$, which we have shown must all be at least $p_0(\epsilon, C)$. Set $N \defeq \prod_{p\in\mathcal P_{\psi, \theta}}p^{k_p}$. Then most pairs $(v,w)\in\mathcal E$ (with respect to $\mu_{\psi, \theta}^{f,g}$) satisfy $\abs{\nu_p(v/N)} + \abs{\nu_p(w/N)}\leq 1$ for all primes $p$.
    Indeed, let 
    \begin{align}
        \mathcal E^{*} \defeq \{ (v,w)\in\mathcal E : \text{for all primes $p$, } \abs{\nu_p(v/N)} + \abs{\nu_p(w/N)}\leq 1  \}.
    \end{align}
    Then by \eqref{eq:concentration_kp} and a union bound, 
    \begin{align}
        \mu_{\psi,\theta}^{f,g}(\mathcal E\setminus\mathcal E^{*}) =
        \mu_{\psi,\theta}^{f,g}( \{ (v, w)\in\mathcal E:\exists p\in\mathcal P_{\psi,\theta} \st \abs{\nu_p(v)-k_p}+\abs{\nu_p(w)-k_p} \geq 2 \} )\\
        \ll_{\epsilon, C} \mu_{\psi,\theta}^{f,g}( \mathcal E )\sum_{p>p_{0}(\epsilon, C)}(p^{-1-2 \epsilon}+p^{-\frac{11}{10}}).
    \end{align}
    Therefore, since $p_0(\epsilon, C)$ is large, $\mu_{\psi,\theta}^{f,g}( \mathcal E^{*} ) \geq \frac{1}{2}\mu_{\psi,\theta}^{f,g}( \mathcal E )$. Taking $V^{*} \defeq \mathcal E^{*}\vert_{V}$ and $W^{*} = \mathcal E^{*}\vert_{W}$, it is then clear that 
    \begin{align}
        \mu_{\psi,\theta}^{f,g}( \mathcal E^{*} ) > \frac{\mu_{\psi,\theta}^{f,g}( \mathcal E )}{2}
        &> \frac{1}{2}\cdot (100e^{C})^{P_{\psi,\theta}(\epsilon, C)}(\Log t)^{\frac{1}{2}(e^{40C}-1)}(\mu_\psi^{f}(V)\mu_\theta^{g}(W)e^{-CK})^{\frac{1}{q'}}\\
        &\geq \frac{1}{2}\cdot (100e^{C})^{P_{\psi,\theta}(\epsilon, C)}(\Log t)^{\frac{1}{2}(e^{40C}-1)}(\mu_\psi^{f}(V^{*})\mu_\theta^{g}(W^{*})e^{-CK})^{\frac{1}{q'}}.
    \end{align}
    Thus $\mathcal E^{*}$ satisfies properties (\ref{itm:is-near-counterexample}) and (\ref{itm:arithmetically-structured}) from Proposition \ref{prop:structure-of-counterexample}. To obtain the combinatorial structure of property (\ref{itm:combinatorially-structured}) we need one final trick: define $\mathcal E'\subseteq \mathcal E^*$ to be minimal with respect to inclusion satisfying property (\ref{itm:is-near-counterexample}). By construction, $\mathcal E'$ must exist, since $\mathcal E^{*}$ is finite and satisfies property (\ref{itm:is-near-counterexample}). Moreover, $\mathcal E'$ satisfies property (\ref{itm:arithmetically-structured}) because $\mathcal E^{*}$ satisfies property (\ref{itm:arithmetically-structured}) itself. It remains to demonstrate property (\ref{itm:combinatorially-structured}).

    Let $V'\defeq \mathcal E'\vert_V$ and $W'\defeq\mathcal E'\vert_W$, and suppose for a contradiction that there exists some $v\in V'$ for which 
    \begin{align}
        \label{eq:failure-of-combinatorial-structure}
        \mu_{\theta}^{g}(\Gamma_{\mathcal E'}(v)) < \frac{1}{q'}\frac{\mu_{\psi,\theta}^{f,g}( \mathcal E' )}{\mu_{\psi}^{f}(V')}.
    \end{align}
    We will show that by removing $v$ we obtain a set that still satisfies property (\ref{itm:is-near-counterexample}) of Proposition \ref{prop:structure-of-counterexample}, contradicting minimality. Indeed, let 
    \begin{align}
        \mathcal E^{\text{new}} \defeq \mathcal E'\cap((V'\setminus\{v\})\times W'),
    \end{align}
    and define $V^{\text{new}}\defeq \mathcal E^{\text{new}}\vert_V$ and $W^{\text{new}}\defeq \mathcal E^{\text{new}}\vert_W$ as usual. First, observe that $\mathcal E^{\text{new}}$ is indeed smaller than $\mathcal E'$, as $v\in V'$ implies that there is $w\in W'$ such that $(v, w) \in\mathcal E'$, but now $(v,w)\not\in\mathcal E^{\text{new}}$.
    Next, from the definition and the assumption \eqref{eq:failure-of-combinatorial-structure}, we obtain
    \begin{align}
        \mu_{\psi,\theta}^{f,g}( \mathcal E^{\text{new}} )
        &= \mu_{\psi,\theta}^{f,g}( \mathcal E' ) -\mu_\psi^{f}(v)\mu_\theta^{g}(\Gamma_{\mathcal E'}(v))
        > \mu_{\psi,\theta}^{f,g}( \mathcal E' )\Big( 1-\frac{1}{q'}\frac{\mu_{\psi}^{f}(v)}{\mu_{\psi}^{f}(V')}  \Big)\\
        &\geq \mu_{\psi,\theta}^{f,g}( \mathcal E' )\Big( 1-\frac{\mu_{\psi}^{f}(v)}{\mu_{\psi}^{f}(V')}  \Big)^{\frac{1}{q'}}
        =\mu_{\psi,\theta}^{f,g}( \mathcal E' )\Big( \frac{\mu_{\psi}^{f}(V^{\text{new}})}{\mu_{\psi}^{f}(V')} \Big)^{\frac{1}{q'}}\\
        &> \frac{1}{2}\cdot (100e^{C})^{P_{\psi,\theta}(\epsilon, C)}(\Log t)^{\frac{1}{2}(e^{40C}-1)}(\mu_\psi^{f}(V^{\text{new}})\mu_\theta^{g}(W^{\text{new}})e^{-CK})^{\frac{1}{q'}},
    \end{align}
    where the last inequality uses the fact that $\mathcal E'$ satisfies property (\ref{itm:is-near-counterexample}) and 
    \begin{align}
        \frac{\mu_{\theta}^{g}(W^{\text{new}})}{\mu_{\theta}^{g}(W')}\leq 1.
    \end{align}
    This contradicts the fact that $\mathcal E'$ is minimal satisfying property (\ref{itm:is-near-counterexample}).
    The case where the contradiction arises from a $w\in W'$ is identical, so $\mathcal E'$ must satisfy property (\ref{itm:is-near-counterexample}).
    This concludes the proof of Proposition \ref{prop:structure-of-counterexample}.
\end{proof}

\section{Anatomy lemmas}
\label{sec:anatomy_lemmas}
\begin{lemma}
    [Unweighted anatomy property]
    \label{lm:unweighted_anatomy}
    For any real $x, t\geq 1, K\in \R$ and $C>0$, 
    \begin{align}
         \#\{ n\leq x: \lvert\{ p\leq t: p\divides n \}\rvert \geq K \} 
        \ll_C xe^{-CK}(\Log t)^{e^{C}-1}.
    \end{align}
\end{lemma}
\begin{proof}
    Take $C > 0$. Then by a Rankin trick
    \begin{align}
        \# \{ n\leq x: \lvert\{ p\leq t: p\divides n \}\rvert \geq K \} 
        \leq e^{-CK}\sum_{n\leq x}\exp \Big( \sum_{\substack{p\leq t \\ p\divides n}} C \Big).
    \end{align}
    Write $f(n) = \exp \Big( \sum_{\substack{p\leq t \\ p\divides n}} C \Big)$, then $f$ is multiplicative and $f(p^{a})\leq e^{C}$ for all prime powers $p^{a}$. Therefore, \cite[Theorem 14.2]{K19} gives the bound 
    \begin{align}
        \sum_{n\leq x}f(n)
        \ll_{C} x\cdot \exp \Big( \sum_{p\leq x}\frac{f(p) - 1}{p} \Big)
        \leq x\cdot \exp \Big( \sum_{p\leq t}\frac{e^{C} - 1}{p} \Big)
        \ll x(\Log t)^{e^{C}-1},
    \end{align}
    where the last inequality is Merten's second estimate \cite[Theorem 3.4]{K19}.
    The lemma now follows immediately.
\end{proof}

\begin{lemma}
    [Divisor anatomy property]
    \label{lm:divisor_anatomy}
    Let $f$ be a multiplicative function that satisfies \eqref{eq:f-condition}.
    Then for any $M\in\N$, real $t\geq 1, K\in\R$ and $C>0$,
    \begin{align}
        \sum_{\substack{mn=M \\ \lvert\{ p\leq t: p\divides m \}\rvert \geq K}} f(n)
        \ll Me^{-CK}(\Log t)^{e^{C}-1}.
    \end{align}
\end{lemma}
\begin{proof}
    Take $C > 0$.
    Then 
    \begin{align}
        \sum_{\substack{mn=M \\ \lvert\{ p\leq t: p\divides m \}\rvert \geq K}} f(n)
        &\leq e^{-CK}\sum_{mn = M}\exp \Big( \sum_{\substack{p\leq t \\ p\divides M}}C \Big)f(n) \\
        &= e^{-CK}
            \prod_{\substack{p\leq t \\ p\divides M}}\Big( (1\star f)(p^{\nu_{p}(M)}) + (e^{C}-1)(1\star f)(p^{\nu_{p}(M)-1}) \Big)
            \prod_{\substack{p>t \\ p\divides M}}\Big((1\star f)(p^{\nu_{p}(M)})\Big)
        \\
        &\leq e^{-CK}\prod_{p\divides M}\Big(p^{\nu_p(M)}\Big)\prod_{\substack{p\leq t\\ p\divides M}}\Big(1 + \frac{e^{C}-1}{p}\Big)
        = Me^{-CK}\prod_{\substack{p\leq t\\ p\divides M}}\Big(1 + \frac{e^{C}-1}{p}\Big)
    \end{align}
    where the last inequality follows from \eqref{eq:f-condition}.
    Finally, the lemma follows by Merten's second estimate since 
    \begin{align}
        \prod_{\substack{p\leq t\\ p\divides M}}\Big(1 + \frac{e^{C}-1}{p}\Big)
        \leq\exp \Big( \sum_{p\leq t}\frac{e^{C} - 1}{p} \Big)
        \ll (\Log t)^{e^{C}-1}.
    \end{align}
\end{proof}

\section{Proof of Proposition \ref{prop:resolution-of-counterexample}}
\label{sec:proof_of_resolution}
We follow the same argument as in \cite[Section 5]{HVW24}.
Let $\psi, \theta, \epsilon, C, t, K, f$ and $g$ be as in the statement of Proposition \ref{prop:resolution-of-counterexample}, and $N$ be the natural number from property (\ref{itm:arithmetically-structured}) of Proposition \ref{prop:structure-of-counterexample} that $\mathcal E'$ satisfies. We will initially prove 
\begin{align}
    \label{eq:almost-contradiction}
    \mu_{\psi,\theta}^{f,g}( \mathcal E' ) \ll_C (\Log t)^{\frac{1}{2}(e^{40C}-1)}(\mu_\psi^{f}(V')\mu_\theta^{g}(W')e^{-10CK})^{\frac{1}{2}},
\end{align}
where $V' \defeq \mathcal E'\vert_V$ and $W'\defeq\mathcal E'\vert_W$.
Since $p_0(\epsilon, C)$ is sufficiently large, this is nearly in contradiction to property (\ref{itm:is-near-counterexample}) of Proposition \ref{prop:structure-of-counterexample}, except that the exponent is $\frac{1}{2}$ rather than $\frac{1}{2}+ \epsilon$. There is a short final argument to deal with this point.
\begin{proof}
    [Proof of \eqref{eq:almost-contradiction}]
    For each $(v, w) \in V'\times W'$, we define 
    \begin{align}
        &v^{-} \defeq \prod_{p:\ \nu_p(v/N)=-1}p,\qquad
        &v^{+} \defeq \prod_{p:\ \nu_p(v/N)=1}p,\\
        &w^{-} \defeq \prod_{p:\ \nu_p(w/N)=-1}p,\qquad
        &w^{+} \defeq \prod_{p:\ \nu_p(w/N)=1}p.
    \end{align}
    Since $\mathcal E'$ satisfies property (\ref{itm:arithmetically-structured}), if $(v,w)\in\mathcal E'$ then all four of $v^{-}, v^{+}, w^{-}$ and $w^{+}$ are coprime, and we may write 
    \begin{align}
        v = N \frac{v^{+}}{v^{-}},\qquad
        w = N \frac{w^{+}}{w^{-}}.
    \end{align}
    Therefore, as $(v,w)\in\mathcal E_{\psi, \theta}^{f,g}$, the condition $D_{\psi, \theta}(v,w)\leq 1$ implies 
    \begin{align}
        \psi(v)\leq \frac{\gcd(v,w)}{w} = \frac{1}{v^{-}w^{+}},\qquad
        \theta(w)\leq \frac{\gcd(v,w)}{v}= \frac{1}{v^{+}w^{-}}.
    \end{align}
    In particular, if we let $w_0\in W'$ maximise $w_0^{+}$, and let $v_0(w)\in\Gamma_{\mathcal E'}(w)$ maximise $v_0^{+}$ over $\Gamma_{\mathcal E'}(w)$, we have 
    \begin{align}
        \psi(v)\leq \frac{1}{v^{-}w_0^{+}},\qquad
        \theta(w)\leq \frac{1}{v_0^{+}(w)w^{-}},
    \end{align}
    for all $(v,w)\in\mathcal E'$ with $v\in\Gamma_{\mathcal E'}(w_0)$.
    By applying property (\ref{itm:combinatorially-structured}) twice and using the bounds above, 
    \begin{align}
        \label{eq:resolution-initial-bound}
        \frac{\mu_{\psi,\theta}^{f,g}( \mathcal E' )}{(\mu_{\psi}^{f}(V')\mu_{\theta}^{g}(W'))^{\frac{1}{2}}}
        &\leq (q')^{\frac{1}{2}}\Big(\frac{\mu_{\psi,\theta}^{f,g}( \mathcal E' )}{\mu_{\psi}^{f}(V')}\Big)^{\frac{1}{2}}\mu_{\psi}^{f}(\Gamma_{\mathcal E'}(w_0))^{\frac{1}{2}}\\
        &\leq q'\Big(\sum_{v\in\Gamma_{\mathcal E'}(w_0)}\mu_{\psi}^{f}(v)\mu_{\theta}^{g}(\Gamma_{\mathcal E'}(v))\Big)^{\frac{1}{2}}\\
        &=q'\Big(\sum_{v\in\Gamma_{\mathcal E'}(w_0)}\frac{f(v)\psi(v)}{v}\sum_{w\in\Gamma_{\mathcal E'}(v)}\frac{g(w)\theta(w)}{w}\Big)^{\frac{1}{2}}\\
        &\leq q'\Big( \frac{1}{w^{+}_0}\sum_{w\in W'}\frac{g(w)}{ww^{-}}\cdot \frac{1}{v_0^{+}(w)}\sum_{v\in\Gamma_{\mathcal E'}(w)}\frac{f(v)}{vv^{-}}  \Big)^{\frac{1}{2}}.
    \end{align}
    The sum above is over all pairs $(v,w)\in\mathcal E'\subseteq\mathcal E_{\psi,\theta}^{f,g}$. Therefore, for all such pairs we have $\frac{vw}{\gcd(v,w)^{2}} = v^{-}v^{+}w^{-}w^{+}$ and hence the condition $\omega_t(v,w)\geq K$ becomes 
    \begin{align}
        \abs{\{ p\leq t :\ p\divides  v^{-}v^{+}w^{-}w^{+}\}}\geq K.
    \end{align}
    It follows that either $\abs{\{ p\leq t :\ p\divides  v^{-}\}}\geq K/4$ or similarly with $p\divides v^{+}, p\divides w^{-}$ or $p\divides w^{+}$.
    Therefore, 
    \begin{align}
        \label{eq:bound_split}
        \frac{\mu_{\psi,\theta}^{f,g}( \mathcal E' )}{(\mu_{\psi}^{f}(V')\mu_{\theta}^{g}(W'))^{\frac{1}{2}}}
        \leq q'(S_1 + S_2 + S_3 + S_4)^{\frac{1}{2}}
    \end{align}
    where 
    \begin{align}
        \label{eq:def_S1}
        S_1 \defeq \frac{1}{w^{+}_0}\sum_{w\in W'}\frac{g(w)}{ww^{-}}\cdot \frac{1}{v_0^{+}(w)}\sum_{\substack{v\in\Gamma_{\mathcal E'}(w)  \\ \abs{\{ p\leq t :\ p\divides  v^{-}\}}\geq K/4}}\frac{f(v)}{vv^{-}}
    \end{align}
    and $S_2$, $S_3$ and $S_4$ are similar expressions where the anatomy condition is placed on $p\divides v^{+}, p\divides w^{-}$ and $p\divides w^{+}$ respectively. We now bound each of these expressions separately.

    \textbf{Bounding $S_1$.}
    We bound the inner sum from \eqref{eq:def_S1}. Since $v = v^{+}\cdot \frac{N}{v^{-}}$ and $v^{+}\leq v_0^{+}(w)$, we can reparametrise $v$ in terms of variables $v^{+}$ and $v^{-}$ and obtain 
    \begin{align}
        \label{eq:bound_inner_S1}
        \sum_{\substack{v\in\Gamma_{\mathcal E'}(w)  \\ \abs{\{ p\leq t :\ p\divides  v^{-}\}}\geq K/4}}\frac{f(v)}{vv^{-}}
        \leq \sum_{v^{+}\leq v_0^{+}(w)}\sum_{\substack{v^{-}\divides N  \\ \gcd(v^{-}, v^{+}) = 1 \\ \abs{\{ p\leq t :\ p\divides  v^{-}\}}\geq K/4}}\frac{f(v^{+}\cdot \frac{N}{v^{-}})}{Nv^{+}}.
    \end{align}
    We factorise the $f$ term to remove the influence of $v^{+}$. For notational convenience, write $N_{v^+}$ for the coprime part of $N$ to $v^{+}$. That is,
    \begin{align}
        N_{v^+} \defeq \prod_{p \notdivides v^{+}}p^{\nu_p(N)}.
    \end{align}
    Then 
    \begin{align}
        \frac{f(v^{+}\cdot \frac{N}{v^{-}})}{Nv^{+}} = \frac{f(v^{+}\frac{N}{N_{v^+}}\cdot \frac{N_{v^+}}{v^{-}})}{v^{+}\frac{N}{N_{v^+}}\cdot N_{v^+}} = \frac{f(v^{+}\frac{N}{N_{v^+}})}{v^{+}\frac{N}{N_{v^+}}}\cdot \frac{f(\frac{N_{v^+}}{v^{-}})}{N_{v^+}}\leq\frac{f(\frac{N_{v^+}}{v^{-}})}{N_{v^+}}.
    \end{align}
    The right-hand side of \eqref{eq:bound_inner_S1} then becomes 
    \begin{align}
        &\leq \sum_{v^{+}\leq v_0^{+}(w)} \Big( \frac{1}{N_{v^+}}\sum_{\substack{v^{-}\divides N_{v^+} \\ \abs{\{ p\leq t :\ p\divides  v^{-}\}}\geq K/4}}f(\frac{N_{v^+}}{v^{-}})  \Big)\\
        &\leq \sum_{v^{+}\leq v_0^{+}(w)} \Big( \frac{1}{N_{v^+}}\sum_{\substack{mn= N_{v^+} \\ \abs{\{ p\leq t :\ p\divides  m\}}\geq K/4}}f(n)  \Big) \ll v_{0}^{+}(w)e^{-10CK}(\Log t)^{e^{40C}-1},
    \end{align}
    by Lemma \ref{lm:divisor_anatomy}. From the definition of $S_1$ we now have 
    \begin{align}
        \label{eq:S1_intermediate_bound}
        S_1
        \ll e^{-10CK}(\Log t)^{e^{40C}-1}\cdot \frac{1}{w_0^{+}}\sum_{w\in W'} \frac{g(w)}{ww^{-}}.
    \end{align}
    This inner sum can be dealt with in a similar way to \eqref{eq:bound_inner_S1}. Writing $N_{w^+}$ now for the coprime part of $N$ to $w^{+}$, 
    \begin{align}
        \sum_{w\in W'}\frac{g(w)}{ww^{-}} \leq \sum_{w^{+}\leq w_0^{+}}\Big(   \frac{1}{N_{w^+}}\sum_{mn = N_{w^+}}g(n)\Big)\leq w_0^{+}.
    \end{align}
    In this case we have used the fact that $(1\star g)(n)\leq n$ instead of Lemma \ref{lm:divisor_anatomy}. Hence $S_1 \ll e^{-10CK}(\Log t)^{e^{40C}-1}$.
    
    \textbf{Bounding $S_2$.}
    We have 
    \begin{align}
        S_2 \defeq \frac{1}{w^{+}_0}\sum_{w\in W'}\frac{g(w)}{ww^{-}}\cdot \frac{1}{v_0^{+}(w)}\sum_{\substack{v\in\Gamma_{\mathcal E'}(w)  \\ \abs{\{ p\leq t :\ p\divides  v^{+}\}}\geq K/4}}\frac{f(v)}{vv^{-}}
    \end{align}
    We use a similar strategy as above to deal with the inner sum,
    \begin{align}
        \sum_{\substack{v\in\Gamma_{\mathcal E'}(w)  \\ \abs{\{ p\leq t :\ p\divides  v^{+}\}}\geq K/4}}\frac{f(v)}{vv^{-}}
        &\leq \sum_{\substack{v^{+}\leq v_0^{+}(w)\\ \abs{\{ p\leq t :\ p\divides  v^{+}\}}\geq K/4 }}\sum_{\substack{v^{-}\divides N\\\gcd(v^{-},v^{+})=1}} \frac{f(v^{+}\cdot \frac{N}{v^{-}})}{Nv^{+}}\\
        &\leq \sum_{\substack{v^{+}\leq v_0^{+}(w)\\ \abs{\{ p\leq t :\ p\divides  v^{+}\}}\geq K/4 }}\Big( \frac{1}{N_{v^+}}\sum_{v^{-}\divides N_{v^+}}f(\frac{N_{v^+}}{v^{-}}) \Big)\\
        &\leq\sum_{\substack{v^{+}\leq v_0^{+}(w)\\ \abs{\{ p\leq t :\ p\divides  v^{+}\}}\geq K/4 }} 1 \ll_C v^{+}_0(w)e^{-10CK}(\Log t)^{e^{40C}-1}.
    \end{align}
    Here the last step uses Lemma \ref{lm:unweighted_anatomy} instead of \ref{lm:divisor_anatomy}.
    The bound $S_2 \ll_Ce^{-10CK}(\Log t)^{e^{40C}-1}$ then follows as in the $S_1$ case.

    \textbf{Bounding $S_3$ and $S_4$.}
    The bound $S_3\ll e^{-10CK}(\Log t)^{e^{40C}-1}$ follows by an analogous argument to the one for bounding $S_1$, with the non-trivial application of Lemma \ref{lm:divisor_anatomy} now made on the $w$ sum instead of the $v $ sum; the bound $S_4\ll_C e^{-10CK}(\Log t)^{e^{40C}-1}$ follows by an analogous argument to the one for bounding $S_2$, with the non-trivial application of Lemma \ref{lm:unweighted_anatomy} made on the $w$ sum instead of the $v$ sum.

    Therefore, $S_1 + S_2+S_3+S_4\ll_C e^{-10CK}(\Log t)^{e^{40C}-1}$. Substituting this bound into \eqref{eq:bound_split} and using $q'<2$, we have
    \begin{align}
        \frac{\mu_{\psi,\theta}^{f,g}( \mathcal E' )}{(\mu_{\psi}^{f}(V')\mu_{\theta}^{g}(W'))^{\frac{1}{2}}}
        \ll_C (e^{-10CK}(\Log t)^{e^{40C}-1})^{\frac{1}{2}},
    \end{align}
    resolving \eqref{eq:almost-contradiction}.
\end{proof}

\begin{proof}
    [Proof of Proposition \ref{prop:resolution-of-counterexample}]
    Let $\mathcal E'$ be as in the statement.
    We will show that, provided that $p_{0}(\epsilon, C)$ is sufficiently large, then 
    $\mathcal E'$ does not satisfy property (\ref{itm:is-near-counterexample}) of Proposition \ref{prop:structure-of-counterexample}, concluding the proof of Proposition \ref{prop:resolution-of-counterexample}.
    First observe that when $K<0$, the conditions of Proposition \ref{prop:resolution-of-counterexample} are the same as if $K=0$. However, setting $K=0$ in this case strengthens the conclusion that $\mathcal E'$ does not satisfy property (\ref{itm:is-near-counterexample}). Therefore, without loss of generality we may assume that $K\geq 0$.
    We also suppose that $p_0(\epsilon, C)$, and therefore $P_{\psi, \theta}(\epsilon, C)$, is sufficiently large.

    In the case that
    \begin{align}
        \label{eq:conclusion_easy_case}
        (\mu_\psi^{f}(V')\mu_\theta^{g}(W'))^{\frac{1}{q}}
        \leq \frac{1}{2}(100e^{C})^{P_{\psi, \theta}(\epsilon, C)}(\Log t)^{\frac{1}{2}(e^{40C}-1)}e^{-\frac{CK}{q'}},
    \end{align}
    then the trivial bound
    \begin{align}
        \mu_{\psi,\theta}^{f,g}( \mathcal E' ) \leq \mu_\psi^{f}(V')\mu_\theta^{g}(W')
    \end{align}
    immediately implies that $\mathcal E'$ does not satisfy property (\ref{itm:is-near-counterexample}) of Proposition \ref{prop:structure-of-counterexample}, concluding the proof. So we may assume that \eqref{eq:conclusion_easy_case} does not hold.

    Using \eqref{eq:almost-contradiction}, for some constant $M_C$ depending only on $C$, 
    \begin{align}
        \mu_{\psi,\theta}^{f,g}( \mathcal E' )
        &\leq M_C(\Log t)^{\frac{1}{2}(e^{40C}-1)}(\mu_\psi^{f}(V')\mu_\theta^{g}(W')e^{-10KC})^{\frac{1}{2}} \\
        &\leq M_C(\Log t)^{\frac{1}{2}(e^{40C}-1)}(\mu_\psi^{f}(V')\mu_\theta^{g}(W')e^{-10KC})^{\frac{1}{2}}\Big( \frac{2^{q}e^{\frac{CKq}{q'}}(\Log t)^{-q(e^{40C}-1)}}{(100e^{C})^{qP_{\psi,\theta}(\epsilon, C)}}\mu_\psi^{f}(V')\mu_\theta^{g}(W') \Big)^{\epsilon}\\
        &\leq (\Log t)^{\frac{1}{2}(e^{40C}-1)}(\mu_\psi^{f}(V')\mu_\theta^{g}(W')e^{-KC})^{\frac{1}{q'}}\cdot e^{CK(\frac{1}{q'}+\epsilon \frac{q}{q'}-5)}\\
        &\leq (\Log t)^{\frac{1}{2}(e^{40C}-1)}(\mu_\psi^{f}(V')\mu_\theta^{g}(W')e^{-KC})^{\frac{1}{q'}}
    \end{align} 
    since $P_{\psi ,\theta}$ can be chosen sufficiently large and $\epsilon \leq 2/5$ by a short computation. This also proves that $\mathcal E'$ does not satisfy property (\ref{itm:is-near-counterexample}) of Proposition \ref{prop:structure-of-counterexample}, concluding the whole argument.
\end{proof}

\bibliographystyle{plain}
\bibliography{bibliography}

\end{document}